\documentclass[a4paper,11pt]{article}
\usepackage[]{graphicx}
\usepackage[]{amsmath,amssymb}
\usepackage[]{amsthm,enumerate}
\usepackage{verbatim,bm}
\usepackage{color}
\usepackage{a4wide}

\newtheorem{theorem}{Theorem}[section]
\newtheorem{lemma}[theorem]{Lemma}

\theoremstyle{definition}

\newtheorem{example}[theorem]{Example}

\theoremstyle{remark}
\newtheorem{remark}[theorem]{Remark}

\title{Symmetric Bush-type generalized Hadamard matrices and association schemes}
\date{\today}

\author{
 Hadi Kharaghani\thanks{Department of Mathematics and Computer Science, University of Lethbridge,
Lethbridge, Alberta, T1K 3M4, Canada.  \texttt{kharaghani@uleth.ca}} 
\and  
 Sho Suda\thanks{Department of Mathematics Education,  Aichi University of Education, 1 Hirosawa, Igaya-cho, Kariya, Aichi 448-8542, Japan. \texttt{suda@auecc.aichi-edu.ac.jp}}}

\begin{document}
\maketitle
\abstract{We define Bush-type generalized Hadamard matrices over abelian groups and construct symmetric Bush-type generalized Hadamard matrices over the additive group of finite field $\mathbb{F}_q$, $q$ a prime power. We then show and study an association scheme obtained from such generalized Hadamard matrices.


}

\section{Introduction}
A {\it Hadamard matrix} $H$ of order $n$ is a square matrix of order $n$ with entries from $\{1,-1\}$ such that $H H^T=nI_n$, where $H^T$ is the transpose of $H$ and $I_n$ is the identity matrix of order $n$. 
A Hadamard matrix $H=(H_{ij})_{i,j=1}^{2n}$ of order $4n^2$, where each $H_{ij}$ is a square matrix of order $2n$, is said to be of {\it Bush-type} if  $H_{ii}=J_{2n}$ for any $i\in\{1,\ldots,2n\}$, and $H_{ij}J_{2n}=J_{2n}H_{ij}=0$ for any distinct $i,j\in\{1,\ldots,2n\}$, where $J_{2n}$ is the square matrix of order $2n$ with all one's entries. 
Bush-type Hadamard matrices have been studied in \cite{B,GC,K,MX,W}. 
In particular, it was shown in \cite{GC} that the existence of symmetric Bush-type Hadamard matrices is equivalent to that of certain symmetric association schemes of class $3$.  
Furthermore, the association scheme of class $3$ is related to strongly regular graphs with certain parameters with the property that the vertex set is decomposed into maximal cliques attaining Delsarete-Hoffmann's bound \cite{HT} (see also \cite[Lemma~1.1]{MX}).



If $\{1,-1\}$ is regarded as a multiplicative group, then one may consider a {\it generalized Hadamard matrix} as a  matrix with entries in a finite abelian group with a multiplicative property, see Section~\ref{sec:gh}. 


In this paper, we define Bush-type generalized Hadamard matrices over an abelian group and demonstrate 
 a construction method by using some special generalized Hadamard matrices over the additive group of finite field $\mathbb{F}_q$, $q$ a prime power and Latin squares obtained from the same field. 
In particular, we focus on the symmetric Bush-type generalized Hadamard matrices of order $q^2$ over the additive group of $\mathbb{F}_q$,  $q$ is a prime power.



We will see that some symmetric association schemes having  interesting properties are obtained from these matrices with a linear map from $\mathbb{F}_q$ into a subfield preserving addition and describe the eigenmatrices by the Kloosterman sums. 
The association scheme can be regarded as an extension of the association schemes of class $3$ obtained from symmetric Bush-type Hadamard matrices, and it is shown to be a fission scheme of the scheme for the case $n=2$ in \cite[Theorem~1]{CD}. 
In particular, our scheme is a fission scheme of the Hamming association scheme $H(2,q)$.


For the case where $q=2^m$ or $q=3^m$ and a linear map being the absolute trace, the eigenmatrices of the association schemes are given explicitly by calculating the Kloosterman sums.
Furthermore,  the association scheme for  the case $q=3^m$ is used to provide an affirmative answer to a recent question raised by Leopardi.


\section{Preliminaries}
\subsection{Generalized Hadamard matrices}\label{sec:gh}
Let $G$ be an additively written finite abelian  group of order $g$.  
A square matrix $H=(h_{ij})_{i,j=1}^{g\lambda}$ of order $g\lambda$ with entries from $G$ is called a {\it generalized Hadamard matrix} with the parameters $(g,\lambda)$ (or $GH(g,\lambda)$) over $G$ 
if for all distinct $i,k\in\{1,2,\ldots,g\lambda\}$, the multiset $\{h_{ij}-h_{kj}: 1\leq j\leq g\lambda\}$ contains exactly $\lambda$ times of each element of $G$. 
The matrix $H$ is {\it normalized} if  $H$ has the first row consisting of $0$, where $0$ denotes the identity element of $G$. 
Any generalized Hadamard matrix is transformed to be normalized by adding suitable elements to the columns. 

Let $H$ be a generalized Hadamard matrix over $G$ with the parameters $(g,\lambda)$. 
Assume that $g\lambda$ is a square, say $k^2$ with $k$ a positive integer.
The matrix $H$ is of {\it Bush-type} if 
$H$ is represented as a block matrix $H=(H_{ij})_{i,j=1}^k$, where each $H_{ij}$ is a square matrix of order $k$, such that $H_{ii}$ is the zero matrix for any $i\in\{1,\ldots,k\}$, and each row or column of  $H_{ij}$  has all entries of $G$ appearing the same number of times for any distinct $i,j\in\{1,\ldots,k\}$. 
Let $G'$ be a finite abelian group of order $g'$, and $\varphi$  a surjective homomorphism from $G$ to $G'$. 
It is easy to see that the matrix $\varphi(H):=(\varphi(h_{ij}))_{i,j=1}^{g\lambda}$ is a generalized Hadamard matrix $GH(g',g\lambda^2/{g'})$.

\subsection{Association schemes}
Let $n$ be a positive integer. 
Let $X$ be a finite set and $R_i$ ($i\in\{0,1,\ldots,n\}$) be a nonempty subset of $X\times X$. 
The \emph{adjacency matrix} $A_i$ of the graph with vertex set $X$ and edge set $R_i$ is a $(0,1)$-matrix indexed by $X$ such that $(A_i)_{xy}=1$ if $(x,y)\in R_i$ and $(A_i)_{xy}=0$ otherwise. 
A \emph{symmetric association scheme} of $n$-class is a pair $(X,\{R_i\}_{i=0}^n)$ satisfying the following:
\begin{enumerate}
\item $A_0=I_{|X|}$.
\item $\sum_{i=0}^n A_i = J_{|X|}$.
\item $A_i$ is symmetric for $i\in\{1,\ldots,n\}$.
\item For all $i$, $j$, $A_i A_j$ is a linear combination of $A_0,A_1,\ldots,A_n$.
\end{enumerate}
The vector space spanned by $A_i$'s forms a commutative algebra, denoted by $\mathcal{A}$ and called the \emph{Bose-Mesner algebra} or \emph{adjacency algebra}.
There exists a basis of $\mathcal{A}$ consisting of primitive idempotents, say $E_0=(1/|X|)J_{|X|},E_1,\ldots,E_n$. 
Since  $\{A_0,A_1,\ldots,A_n\}$ and $\{E_0,E_1,\ldots,E_n\}$ are two bases in $\mathcal{A}$, there exist the change-of-bases matrices $P=(p_{ij})_{i,j=0}^n$, $Q=(q_{ij})_{i,j=0}^n$ so that
\begin{align*}
A_j=\sum_{i=0}^n p_{ij}E_j,\quad E_j=\frac{1}{|X|}\sum_{i=0}^n q_{ij}A_j.
\end{align*}
The matrices $P$ or $Q$ are said to be {\it first or second eigenmatrices}. 
An association scheme is said to be {\it self-dual} if $P=Q$ for suitable rearranging the indices of the adjacency matrices and the primitive idempotents.   

The association scheme is a {\it translation association scheme} if the vertex set $X$ has the structure of an additively written abelian group, and
for all $i\in\{0,1,\ldots,n\}$, 
\begin{align*}
(x,y)\in R_i\Rightarrow (x+z,y+z)\in R_i.
\end{align*}
For translation association schemes, the first eigenmatrix is calculated by the characters as follows. 
For $i\in \{0,1,\ldots,n\}$ set $N_i=\{x\in X:(0,x)\in R_i\}$. 
For each character $\chi$ of $X$ we have 
\begin{align*}
A_i\chi=(\sum_{x\in N_i}\chi(x))\chi.
\end{align*}
Letting $X^*$ be the dual group of $X$, set $N_j^*=\{\eta\in X^*:E_j\eta=\eta\}$. 
Then the first eigenmatrix of the translation association scheme is expressed as 
\begin{align*}
P_{ij}=\sum_{x\in N_i}\chi(x) \text{ for } \chi\in N_j^*.
\end{align*}

Let $(X,\{R_i\}_{i=0}^n)$, $(X,\{R'_i\}_{i=0}^{n'})$ be symmetric association schemes. 
If there exists a partition $\Lambda_0:=\{0\},\Lambda_1,\ldots,\Lambda_{n'}$ of $\{0,1,\ldots,n\}$ such that $R'_i=\cup_{j\in\Lambda_i}R_j$ for any $i\in\{0,1,\ldots,n'\}$, then $(X,\{R_i\}_{i=0}^n)$ is said to be {\it fission} of $(X,\{R'_i\}_{i=0}^{n'})$ and $(X,\{R_i'\}_{i=0}^{n'})$ is said to be {\it fusion} of $(X,\{R_i\}_{i=0}^{n})$. 

\section{Construction of Bush-type generalized Hadamard matrices}
In this section, we show a method of construction for the Bush-type generalized Hadamard matrices. The method can be viewed as a generalization of the method used in \cite{K} in which a Hadamard matrix of order $2n$ and a Latin square of order $2n$ were used to construct a Bush-type Hadamard matrix of order $4n^2$.  

Let $H=(h_{ij})_{i,j=1}^{g\lambda}$ be a normalized generalized Hadamard matrix $GH(g,\lambda)$ over a finite abelian group $G$ of order $g$ and $L=(l(i,j))_{i,j=1}^{g\lambda}$ be a Latin square of order $g\lambda$ with entries from $\{1,\ldots,g\lambda\}$. 
We may assume that $L$ has the diagonal entries $1$ after permuting rows and columns appropriately. 
Define a matrix $C_k$ ($k=1,\ldots,g\lambda$) as
\begin{align*}
C_k=(-h_{ki}+h_{kj})_{i,j=1}^{g\lambda}.
\end{align*}
Define $M(H,L)$ as a square block matrix of order $g^2\lambda^2$ with $(i,j)$-block equal to $C_{l(i,j)}$. 
The $((i,i'),(j,j'))$-entry of $M(H,L)$ means the $(i',j')$-entry of $(i,j)$-block of $M$. 
The $(i,i')$-th row of $M$ means the $i'$-th row in the $i$-th block of $M$.  

\begin{theorem}\label{thm:GH}
Let $G$ be an additively written finite abelian group of order $g$. 
Let $H$ be a normalized generalized Hadamard matrix $GH(g,\lambda)$ over $G$, and $L$ be a Latin square of order $g\lambda$ with entries from $\{1,\ldots,g\lambda\}$ and with diagonal entries equal to $1$. 
Let  $M=M(H,L)$ defined above.
Then the following hold.
\begin{enumerate}
\item The matrix $M$ is a Bush-type generalized Hadamard matrix $GH(g,g\lambda^2)$. 
\item The matrix $M$ is symmetric if and only if $-h_{l(i,j)k}+h_{l(i,j)l}=-h_{l(j,i)l}+h_{l(j,i)k}$ for $i,j,k,l\in\{1,\ldots,g\lambda\}$. 
\item Let $G'$ be a finite abelian group of order $g'$, and $\varphi$ be a surjective homomorphism from $G$ to $G'$. Then the generalized Hadamard matrix $\varphi(M)$ is  of Bush-type.
\end{enumerate}
\end{theorem}
\begin{proof}
(i): 
Let $i,i',j,j',k,k'$ be elements of $\{1,\ldots,g\lambda\}$. 
We calculate the difference of the $(i,i')$-row and the $(j,j')$-row of $M$ for $(i,i')\neq (j,j')$. 

In the case where $i=j$, we put $l=l(i,k)$. 
The difference of the $(i',k')$-entry and the $(j',k')$-entry of $C_{l}$ is $-h_{l,j}+h_{l,j'}$. 
This value does not depend on the choice of $k'$, and runs over $G$ when $k$ runs over $\{1,\ldots,g\lambda\}$ since  $H$ is a $GH(g,\lambda)$ and $L$ is a Latin square. 

In the case where $i\neq j$, we put $l_1=l(i,k),l_2=l(j,k)$. 
The difference of the $(i',k')$-entry of $C_{l_1}$ and the $(j',k')$-entry of $C_{l_2}$ is $\gamma+h_{l_1,k'}-h_{l_2,k'}$, where $\gamma=-h_{l_1,i'}+h_{l_2,j'}$.  
This value runs over $G$ when $k'$ runs over $\{1,\ldots,g\lambda\}$ since $H$ is a $GH(g,\lambda)$ and $L$ is a Latin square. 
Thus, $M$ is a generalized Hadamard matrix over $G$. 

All diagonal blocks of $M$ are $C_1$ whose entries are all the identity element in $G$. 
Each off-diagonal block is $C_i$ for some $i\in\{2,\ldots,g\lambda\}$. 
Each row and column of the matrix $C_i$ has entries of $G$, each appearing exactly one time. 
Thus, a generalized Hadamard matrix $M$ is of Bush-type.

(ii): The $(i,j)$-block of $M$ is $C_{l(i,j)}$. Thus, $M$ is symmetric if and only if $C_{l(i,j)}=C_{l(j,i)}^T$ for $i,j\in \{1,\ldots,g\lambda\}$. 
This is equivalent to the condition that $-h_{l(i,j)k}+h_{l(i,j)l}=-h_{l(j,i)l}+h_{l(j,i)k}$ for $i,j,k,l\in\{1,\ldots,g\lambda\}$. 

(iii): The property of Bush-type for $\varphi(M)$ follows from (i). 
\end{proof}

\begin{example}\label{ex:gh}
Let $q$ be a prime power. 
Let $\mathbb{F}_q$ be a finite field with $q$ elements $\alpha_1=0,\alpha_2,\ldots,\alpha_q$.

Let $H$ be the multiplication table of $\mathbb{F}_q$, i.e., the $(i,j)$-entry of $H$ is $\alpha_i\cdot \alpha_j$. 
Then $H$ is a generalized Hadamard matrix $GH(q,1)$ over the additive group of $\mathbb{F}_q$. 

Let $L$ be the subtraction table of $\mathbb{F}_q$, i.e., the $(i,j)$-entry of $L$ is $-\alpha_i+\alpha_j$. 
Then $L$ is a Latin square over $\mathbb{F}_q$ with diagonal entries equal to $0$.

The matrix $M=M(H,L)$ obtained from $H$ and $L$ is a Bush-type generalized Hadamard matrix $GH(q,q)$ over the additive group of $\mathbb{F}_q$ by Theorem~\ref{thm:GH}(i). 
Furthermore, $M$ is symmetric by Theorem~\ref{thm:GH}(ii).

Let $\mathbb{F}_p$ be a subfield of $\mathbb{F}_q$, and let $\varphi$ be a linear map from $\mathbb{F}_q$ to $\mathbb{F}_p$. 
By \cite[Theorem~2.2.4]{LN}, there exists an element $\beta\in\mathbb{F}_q$ such that $\varphi(\alpha)=\text{tr}(\beta\alpha)$, where $\text{tr}$ is  the trace function from $\mathbb{F}_q$ to $\mathbb{F}_p$.
By \cite[Theorems~2.2.3(iii), 2.2.4]{LN} $\varphi$ is surjective. 
Then $\varphi(M)$ is a Bush-type generalized Hadamard matrix $GH(p,q^2/p)$ by Theorem~\ref{thm:GH}(iii).
\end{example}
\section{Association schemes related to symmetric Bush-type Generalized Hadamard matrices over $\mathbb{F}_q$}
We will show that an association scheme is obtained from the Generalized Hadamard matrix $GH(p,q^2/p)$ in Example~\ref{ex:gh}
and continue to use all the notations  in Example~\ref{ex:gh}.

We define permutation matrices $P_{\alpha_i}$ ($i\in \{1,\ldots,q\}$) by splitting 
\begin{align*}
L=\sum_{i=1}^{q} \alpha_i P_{\alpha_i}.
\end{align*}
Note that $P_0=I_q$. 
Since $C_1=0$ holds, we observe that 
\begin{align*}
M&=\sum_{i=1}^q P_{\alpha_i}\otimes C_i=\sum_{i=2}^q P_{\alpha_i}\otimes C_i\\
&=\sum_{i=2}^q P_{\alpha_i}\otimes (\sum_{j=1}^q \alpha_{j}P_{\alpha_i^{-1}\alpha_j})=\sum_{j=1}^q \alpha_j(\sum_{i=2}^q P_{\alpha_i}\otimes P_{\alpha_i^{-1}\alpha_j}).
\end{align*} 
The matrix $\varphi(M)$ is written as follows:
\begin{align}\label{eq:1}
\varphi(M)=\sum_{j=1}^q \varphi(\alpha_j)(\sum_{i=2}^q P_{\alpha_i}\otimes P_{\alpha_i^{-1}\alpha_j}).
\end{align}
Let $\mathbb{F}_p=\{\beta_1=0,\beta_2,\ldots,\beta_p\}$. 
From \eqref{eq:1}, we define $(0,1)$-matrices $A_i$ for $1\leq i\leq p$ by 
\begin{align*}
A_i&=\sum_{y\in \varphi^{-1}(\beta_i)}(\sum_{x\in \mathbb{F}_q^*}P_x\otimes P_{x^{-1}y}), 
\end{align*}
where $\mathbb{F}_q^*=\mathbb{F}_q\setminus\{0\}$. 
Set $A_0=I_{q^2}$ and $A_{p+1}=I_q\otimes (J_{q}-I_q)$. 
Define $X$ and $R_i$ as follows:
\begin{align*} 
X&=\mathbb{F}_q\times\mathbb{F}_q, \\
R_0&=\{(x,x)\mid x\in X\},\\
R_i&=\{((x_1,x_2),(y_1,y_2))\in X\times X\mid x_1\neq y_1, \varphi((-x_1+y_1)(-x_2+y_2))=\beta_i\}\\ 
R_{p+1}&=\{((x_1,x_2),(y_1,y_2))\in X\times X\mid x_1= y_1,x_2\neq y_2\}, 
\end{align*}
where $i=1,\ldots,p$. 
Then $A_i$ is the adjacency matrix of the graph $(X,R_i)$ for each $i=0,1,\ldots,p+1$.

The following easy to see lemma will be used later.
\begin{lemma}\label{lem:1}
Let $V$ be a vector space over $\mathbb{F}_q$ of dimension $n$. 
Let $V_i$ be a subspace of $V$ of dimension $n-1$ for $i=1,2$ such that $\dim(V_1\cap V_2)=n-2$.
Then for any element $a\in \mathbb{F}_q$, the multiset $\{x+y+a\mid x\in V_1,y\in V_2\}$ contains every element of $V$ exactly $q^{n-2}$ times.
\end{lemma}

Define the {\it Kloosterman sum} $K(a,b)$ as 
\begin{align*}
K(a,b)=\sum_{x\in GF(q)^*}\omega^{\text{tr}(ax+bx^{-1})}, 
\end{align*}
where $q=p^n$, $p$ is a prime number and $\omega$ is a primitive $p$-th root of unity.
It is easy to see $K(a,b)$ depends only on $ab$. 
Thus we may denote $K(a,b)$ by $K(ab)$.

We now prove the following which demonstrate that the generalized Hadamard matrices in Example~\ref{ex:gh} lead to  symmetric association schemes.
\begin{theorem}\label{thm:as}
The the following hold.
\begin{enumerate}
\item The pair $(X,\{R_i\}_{i=0}^{p+1})$ defined above is a symmetric association scheme. 
\item The scheme is self-dual.
\item The first eigenmatrix $P$ is given by 
\begin{align*}
P=\begin{pmatrix}
1        & \frac{q(q-1)}{p}\bold{1}_p^T & q-1\\
\bold{1}_p& K&-\bold{1}_p\\
1        & -\frac{q}{p}\bold{1}_p^T & q-1
\end{pmatrix}, 
\end{align*}
where $\bold{1}_p$ denotes the all-ones column vector of length $p$, $K=(\sum_{\gamma\in\varphi^{-1}(\beta_i)}K(\gamma\gamma_j))_{i,j=1}^p$ and $\gamma_j\in\varphi^{-1}(\beta_j)$ for $j=1,\ldots,p$. 
\end{enumerate}
\end{theorem}
\begin{proof}[Proof of Theorem~\ref{thm:as}(i)]
It is easy to see that each $A_i$ is a non-zero symmetric $(0,1)$-matrix and all the sum of them equals to the all-ones matrix. 
Now we are going to prove that the vector space $\mathcal{A}:=\text{span}_{\mathbb{R}}(A_0,\ldots,A_{p+1})$ is closed under the multiplication. 

It is obvious that $A_{p+1}^2\in\mathcal{A}$. 
Since each non-diagonal block of $A_i$ for $i\in\{1,\ldots,p\}$ is a sum of $\frac{q}{p}$ permutation matrices, 
we see that $A_i A_{p+1}=A_{p+1}A_i\in\mathcal{A}$ for $i\in\{1,\ldots,p\}$. 

By $P_x\cdot P_y=P_{x+y}$, we obtain the following equalities:
\begin{align}
A_i A_j&=\sum_{x_1,x_2\in\mathbb{F}_q^*}P_{x_1+x_2}\otimes (\sum_{\substack{y\in\varphi^{-1}(\beta_i)\\w\in\varphi^{-1}(\beta_j)}}P_{x_1^{-1}y+x_2^{-1}w})\notag\\
&=\sum_{x_1\in\mathbb{F}_q^*}P_0\otimes (\sum_{\substack{y\in\varphi^{-1}(\beta_i)\\w\in\varphi^{-1}(\beta_j)}}P_{x_1^{-1}(y-w)})+\sum_{x\in\mathbb{F}_q^*,x_1\in\mathbb{F}_q^*\setminus\{x\}}P_{x}\otimes (\sum_{\substack{y\in\varphi^{-1}(\beta_i)\\w\in\varphi^{-1}(\beta_j)}}P_{x_1^{-1}y+(x-x_1)^{-1}w})\notag\\
&=\sum_{t\in\mathbb{F}_q^*}P_0\otimes (\sum_{\substack{y\in\varphi^{-1}(\beta_i)\\w\in\varphi^{-1}(\beta_j)}}P_{t(y-w)})+\sum_{x\in\mathbb{F}_q^*}P_{x}\otimes \sum_{x_1\in\mathbb{F}_q^*\setminus\{x\}}(\sum_{\substack{y\in\varphi^{-1}(\beta_i)\\w\in\varphi^{-1}(\beta_j)}}P_{x_1^{-1}y+(x-x_1)^{-1}w})\label{eq:11}.
\end{align}

For the first term in \eqref{eq:11} we have the following:
\begin{align*}
\sum_{t\in\mathbb{F}_q^*}P_0\otimes (\sum_{\substack{y\in\varphi^{-1}(\beta_i)\\w\in\varphi^{-1}(\beta_j)}}P_{t(y-w)})
&=\begin{cases} 
\sum_{t\in\mathbb{F}_q^*}P_0\otimes (\frac{q}{p}P_0+\sum_{\substack{y,w\in\varphi^{-1}(\beta_i)\\ y\neq w}}P_{t}) &\text{ if } i=j\\
\sum_{t\in\mathbb{F}_q^*}P_0\otimes (\sum_{\substack{y\in\varphi^{-1}(\beta_i)\\w\in\varphi^{-1}(\beta_j)}}P_{t}) &\text{ if } i\neq j
\end{cases}\\
&=\begin{cases} 
\frac{q(q-1)}{p} I_{q^2}+\frac{q}{p}(\frac{q}{p}-1)I_q\otimes (J_q-I_q) &\text{ if } i=j\\
(\frac{q}{p})^2 I_q\otimes (J_q-I_q) &\text{ if } i\neq j
\end{cases}\\
&=\begin{cases} 
\frac{q(q-1)}{p} A_0+\frac{q}{p}(\frac{q}{p}-1)A_{p+1} &\text{ if } i=j\\
(\frac{q}{p})^2 A_{p+1} &\text{ if } i\neq j.
\end{cases}
\end{align*}

For the second term in \eqref{eq:11} we define 
\begin{align*}
X_{i,j,y,w,x}:=\left\{x\left(\frac{y}{z}+\frac{w}{x-z}\right
)\mid z\in\mathbb{F}_p^*,z\neq x\right\}
\end{align*} 
for $x\in \mathbb{F}_p^*,y\in \varphi(\beta_i),z\in \varphi(\beta_j)$. 
Then $X_{i,j,y,w,x}$ does not depend on the choice of $x$. 
Indeed for $a\in\mathbb{F}_p^*$
\begin{align*}
X_{i,j,y,w,a x}&=\{ax(\frac{y}{z}+\frac{w}{ax-z})\mid z\in\mathbb{F}_p^*,z\neq ax\}\\
&=\{ax(\frac{y}{z'}+\frac{w}{ax-az'})\mid z'\in\mathbb{F}_p^*,z'\neq x\}\\
&=\{x(\frac{y}{z'}+\frac{w}{x-z'})\mid z'\in\mathbb{F}_p^*,z'\neq x\}\\
&=X_{i,j,y,w,x}.
\end{align*}
Thus we may denote $X_{i,j,y,w}=X_{i,j,y,w,x}$. 
Next we determine a multiset $$Y_{i,j}:=\{\frac{y}{z}+\frac{w}{1-z}\mid y\in\varphi^{-1}(\beta_i),w\in\varphi^{-1}(\beta_j)\},$$
where $z\in\mathbb{F}_q^*\setminus\{1\}$.
When $z\in \mathbb{F}_p\setminus\{1\}$, we have the following as multisets
\begin{align*}
Y_{i,j}&=\frac{q}{p}\varphi^{-1}(\frac{\beta_i}{z}+\frac{\beta_j}{1-z}).
\end{align*}
When $z\in \mathbb{F}_q^*\setminus \mathbb{F}_p$, letting $y_0\in \varphi(\beta_i),w_0\in \varphi(\beta_j)$, we set $c=-\frac{y_0}{z}-\frac{w_0}{1-z}$. 
Then, by Lemma~\ref{lem:1}, we have 
\begin{align*}
Y_{i,j}=(\frac{y}{z}+\frac{w}{1-z}+c\mid y,w\in\varphi^{-1}(0))
=\frac{q}{p^2}\mathbb{F}_q.
\end{align*}

We now calculate the second term:
\begin{align*}
\sum_{x\in\mathbb{F}_q^*}P_{x}\otimes &\sum_{x_1\in\mathbb{F}_q^*\setminus\{x\}}(\sum_{\substack{y\in\varphi^{-1}(\beta_i)\\w\in\varphi^{-1}(\beta_j)}}P_{x_1^{-1}y+(x-x_1)^{-1}w})\\
&=\sum_{x\in\mathbb{F}_q^*}P_{x}\otimes \sum_{\substack{y\in\varphi^{-1}(\beta_i)\\w\in\varphi^{-1}(\beta_j)}}\sum_{z\in X_{i,j,y,w}}P_{x^{-1}z}\\
&=\sum_{x\in\mathbb{F}_q^*}P_{x}\otimes \sum_{z\in Y_{i,j}}P_{x^{-1}z}\\
&=\sum_{x\in\mathbb{F}_q^*}P_{x}\otimes \sum_{x_1\in\mathbb{F}_q^*\setminus\{1\}}(\sum_{z\in \varphi^{-1}(\frac{\beta_i}{x_1}+\frac{\beta_j}{1-x_1})}\frac{q}{p}P_{x^{-1}z}+\sum_{z\in \mathbb{F}_q}\frac{q}{p^2}P_{x^{-1}z})\\
&=\frac{q}{p}\sum_{x_1\in\mathbb{F}_q^*\setminus\{1\}}\sum_{z\in \varphi^{-1}(\frac{\beta_i}{x_1}+\frac{\beta_j}{1-x_1})}\sum_{x\in\mathbb{F}_q^*}P_{x}\otimes P_{x^{-1}z}+\frac{q(q-p)}{p^2}(J_q-I_q)\otimes J_q\\
&=\frac{q}{p}\sum_{h}A_h+\frac{q(q-p)}{p^2}(J_{q^2}-A_0-A_1), 
\end{align*}
where $h$ runs over the set such that  $\varphi(\alpha_h)=\frac{\beta_i}{x_1}+\frac{\beta_j}{1-x_1}$ for $x_1\in \mathbb{F}_q^*\setminus\{1\}$ as a multiset.
Therefore \eqref{eq:11} is in $\mathcal{A}$. 
Thus the pair $(X,\{R_i\}_{i=0}^{p+1})$ is a symmetric association scheme. 
\end{proof}

\begin{proof}[Proof of Theorem~\ref{thm:as}(ii)]
The association scheme $(X,\{R_i\}_{i=0}^{p+1})$ is clearly translation. 
The dual of this is $(X^*,\{S_i\}_{i=0}^{p+1})$ such that $X^*$ is the character group of $X$ and 
\begin{align*} 
S_0&=\{(\chi,\chi)\mid \chi\in X^*\},\\
S_i&=\{((\chi_1,\chi_2),(\eta_1,\eta_2))\in X^*\times X^*\mid\varphi(\prod_{i=1}^2(\chi_i(x_i)-\eta_i(x_i)))=\beta_i \text{ for all }(x_1,x_2)\in X\}\\ 
S_{p+1}&=\{((\chi_1,\eta),(\chi_2,\eta))\in X^*\times X^*\mid \chi_1\neq \chi_2 \},
\end{align*}
where $i=1,\ldots,p$ and we $X^*$ is regarded as the direct product of the dual group of $\mathbb{F}_q$. 
The correspondence of $x=(x_1,x_2)$ to $\chi$ with $\chi(y)=\omega^{\text{tr}(\sum_{i}x_i y_i)}$, where $\omega$ is a primitive $q$-th root of unity, gives an isomorphism from the scheme to its dual. 
Therefore  $(X,\{R_i\}_{i=0}^{p+1})$ is self-dual. 
\end{proof}

\begin{proof}[Proof of Theorem~\ref{thm:as}(iii)]
Finally we determine the first eigenmatrix. 
For $N_i$ ($i=0,1,\ldots,p+1$) and $\chi\in N_j^*$ ($j=1,\ldots,p+1$), we calculate $\sum_{x\in N_i}\chi(x)$ as follows.
For $i=0$, it is easy to see that the first row of $P$ has the desired values.

For $1\leq i,j\leq p$,   
there exist $a,b\in \mathbb{F}_q$ with $a b=\gamma_j\in\varphi^{-1}(\beta_i)$ such that $\chi(x)=\omega^{\text{tr}(ax_1+bx_2)}$ for $x=(x_1,x_2)\in X$. 
Then
\begin{align*}
\sum_{x\in N_i}\chi(x)&=\sum_{(x_1,x_2)\in N_i}\chi(ax_1+bx_2)=\sum_{\gamma\in\varphi^{-1}(\beta_i)}\sum_{t\in \mathbb{F}_q^*}\chi(at+\frac{b\gamma}{t})\\
&=\sum_{\gamma\in\varphi^{-1}(\beta_i)}K(\gamma_j\gamma).  
\end{align*}

For $i=p+1,1\leq j\leq p$,
\begin{align*}
\sum_{x\in N_{p+1}}\chi(x)&=\sum_{(x_1,x_2)\in N_{q+1}}\chi(ax_1+bx_2)=\sum_{t\in \mathbb{F}_q^*}\chi(\frac{b}{t})=\sum_{x\in \mathbb{F}_q^*}\chi(x)=-1.
\end{align*}

For $1\leq i\leq p,j=p+1$, 
there exists $a\in \mathbb{F}_q^*$ such that $\chi(x)=\omega^{\text{tr}(ax_2)}$ for $x=(x_1,x_2)\in X$.  
Then 
\begin{align*}
\sum_{x\in N_{i}}\chi(x)&=\sum_{\gamma\in\varphi^{-1}(\beta_i)}\sum_{(x_1,x_2)\in N_{i}}\chi(ax_1)=\sum_{\gamma\in\varphi^{-1}(\beta_i)}\sum_{t\in \mathbb{F}_q^*}\chi(t)=-\frac{q}{p}.
\end{align*}

For $i,j=p+1$, 
\begin{align*}
\sum_{x\in N_{p+1}}\chi(x)&=\sum_{(x_1,x_2)\in N_{p+1}}\chi(ax_1)=\sum_{t\in \mathbb{F}_q^*}\chi(0)=-q+1.
\end{align*}
Therefore we obtain the desired eigenmatrix.
\end{proof}

\begin{remark}
When we take $\varphi$ as the identity mapping on $\mathbb{F}_q$, 
we obtained a symmetric association scheme of class $q+1$ by Theorem~\ref{thm:as}. 
 In \cite{CD}, de Caen and van Dam gave a $n+q-2$ class fission scheme of the Hamming scheme $H(n,q)$. 
When $n=2$, their scheme of class $q$ is a fusion scheme of our scheme of $q+1$ class by merging $R_1$ and $R_{q+1}$. 
\end{remark}

\section{Applications}
In this section,
we give an explicit formula for the eigenmatrix for the case $q=2,3$ and $\varphi$ being the absolute trace.

\subsection{The case where $q=2^m$}
Let $p=2$, $q=2^m$ and $\omega=-1$ in Theorem~\ref{thm:as}. 
We take $\varphi$ as the absolute trace from $\mathbb{F}_{2^m}$ to $\mathbb{F}_2$.
In this case we obtain the fusion scheme of class $3$ by Theorem~\ref{thm:as}.

The matrix $\varphi(H)$ is a generalized Hadamard matrix over $\mathbb{F}_2$. 
Here we regard $\mathbb{F}_2$ as the additive group on $\{0,1\}$.
If we replace the entries $0,1$ with $1,-1$ respectively, 
then we have a symmetric Bush-type Hadamard matrix (over $\{1,-1\}$ as a multiplicative group). 
Then it is shown in \cite{W} and \cite{HT} that this yields an association scheme of class $3$ with the following first eigenmatrix
\begin{align*}
P=\begin{pmatrix}
1 & 2^{m-1}(2^m-1) & 2^{m-1}(2^m-1) &2^m-1\\
1 & 2^m & -2^{m-1} &-1\\
1 & -2^{m-1} & 2^m &-1\\
1 & -2^{m-1} & -2^{m-1} &2^m-1\\
\end{pmatrix}.
\end{align*}
See  also \cite{GC} for related topics.

\subsection{The case where $q=3^m$}
Let $p=3$, $q=3^m$ and $\omega$ denote a primitive third root of unity in Theorem~\ref{thm:as}. 
We take $\varphi$ as the absolute trace from $\mathbb{F}_{3^m}$ to $\mathbb{F}_3$.  
In this case we obtain the association scheme of class $4$ by Theorem~\ref{thm:as}.
We calculate its eigenmatrices explicitly. 

Define $T_i=\{a\in\mathbb{F}_{3^m}\mid \text{tr}(a)=i\}$ for $i\in\mathbb{F}_3$. 
For $a\in \mathbb{F}_{3^m}$ and $i\in\mathbb{F}_3$, let $S_{a,i}=\sum_{t\in T_i}K(at)$.
\begin{lemma}
The following hold.
\begin{enumerate}
\item $S_{0,i}=-3^{m-1}$ holds for $i=0,1,2$.
\item If $a\neq 0$, then $\sum_{i=0}^2 S_{a,i}=0$.
\item If $a\neq 0$, then $\sum_{i=0}^2 \omega^i S_{a,i}=3^m\omega^{\text{tr}(-1/a)}$.
\item For $a\neq 0$, let $\text{\rm{tr}}(-\frac{1}{a})=i$. Then $S_{a,i}=2\cdot3^{m-1}$ and $S_{a,j}=-3^{m-1}$ ($j\neq i$) hold.
\end{enumerate}
\end{lemma}
\begin{proof}
(i): 
\begin{align*}
S_{0,i}=\sum_{t\in T_i}K(0)=3^{m-1}\sum_{x\in \mathbb{F}_{3^m}^*}\omega^{\text{tr}(x)}=3^{m-1}(\sum_{x\in \mathbb{F}_{3^m}}\omega^{\text{tr}(x)}-\omega^0)=-3^{m-1}.
\end{align*}

(ii): For $a\in \mathbb{F}_{3^m}^*$, 
\begin{align*}
\sum_{i=0}^2S_{a,i}&=\sum_{i=0}^2\sum_{t\in T_i}K(at)=\sum_{t\in \mathbb{F}_{3^m}}K(at)=\sum_{t\in \mathbb{F}_{3^m}}\sum_{x\in \mathbb{F}_{3^m}^*}\omega^{\text{tr}(x+\frac{at}{x})}\\ 
&=\sum_{x\in \mathbb{F}_{3^m}^*}\sum_{t\in \mathbb{F}_{3^m}}\omega^{\text{tr}(x+\frac{at}{x})}=\sum_{x\in \mathbb{F}_{3^m}^*}\sum_{t\in \mathbb{F}_{3^m}}\omega^{\text{tr}(x+\frac{t}{x})}\\
&=\sum_{x\in \mathbb{F}_{3^m}^*}\sum_{t\in \mathbb{F}_{3^m}}\omega^{\text{tr}(t)}=0.
\end{align*}

(iii): 
\begin{align*}
\sum_{i=0}^2\omega^i S_{a,i}&=\sum_{i=0}^2\sum_{t\in T_i}\sum_{x\in \mathbb{F}_{3^m}^*}\omega^{\text{tr}(x+\frac{at}{x})+i}=\sum_{i=0}^2\sum_{t\in T_i}\sum_{x\in \mathbb{F}_{3^m}^*}\omega^{\text{tr}(x+\frac{at}{x}+t)}\\ 
&=\sum_{x\in \mathbb{F}_{3^m}^*}\sum_{t\in \mathbb{F}_{3^m}}\omega^{\text{tr}(x+\frac{at}{x}+t)}\\
&=\sum_{t\in \mathbb{F}_{3^m}}(1+\omega+\omega^2)+\sum_{t\in \mathbb{F}_{3^m}}\omega^{\text{tr}(-\frac{1}{a})}\\
&=3^m\omega^{\text{tr}(-\frac{1}{a})}.
\end{align*}

(iv): Since the Kloosterman sum is a real number, it follows from (ii) and (iii).
\end{proof}

Since we can take $a\in \mathbb{F}_q$ with $\text{tr}(-\frac{1}{a})=i$ for any $i=0,1,2$, we obtain the following theorem. 
\begin{theorem}\label{thm:eigenmat3}
The first eigenmatrix of the $4$-class fusion scheme is given as follows:
\begin{align*}
P=\begin{pmatrix}
1 & 3^{m-1}(3^m-1) & 3^{m-1}(3^m-1) &3^{m-1}(3^m-1) &3^m-1 \\
1 & 2\cdot3^{m-1} & -3^{m-1} &-3^{m-1} &-1 \\
1 & -3^{m-1} & 2\cdot3^{m-1} &-3^{m-1} &-1 \\
1 & -3^{m-1} & -3^{m-1} &2\cdot3^{m-1} &-1 \\
1 & -3^{m-1} & -3^{m-1} & -3^{m-1} &3^m-1 \\
\end{pmatrix}.
\end{align*}
\end{theorem}

\begin{remark}
The scheme of class $4$ has the following properties.
\begin{enumerate}
\item For any $i=1,2,3$, $A_i$ is a strongly regular graph with the parameters $(v,k,\lambda,\mu)=(9^m,3^{m-1}(3^m-1),9^{m-1},3^{m-1}(3^{m-1}-1))$. 
The graph $A_4$ is a disjoint union of cliques of the same size.
Moreover, $A_i+A_4$ is a strongly regular graph with a spread which is described as $A_4$.
\item The binary relations $R_i$'s are given as follows:
\begin{align*}
R_0&=\{(x,x)\mid x\in X\},\\
R_i&=\{(x,y)\mid x=(x_1,x_2),y=(y_1,y_2)\in X, x_1\neq y_1, \text{tr}((-x_1+y_1)(-x_2+y_2))=i\}\\ 
R_{4}&=\{(x,y)\mid x=(x_1,x_2),y=(y_1,y_2)\in X, x_1= y_1,x_2\neq y_2\},
\end{align*}
where $i=1,2,3$. 
A bijection from $X$ to $X$ defined as $(x_1,x_2)\mapsto (2x_1,x_2)$ swaps $A_2$ for $A_3$. 
The graph with adjacency matrix $A_2+A_3$ provides an example to the problem raised in \cite[Question~2]{L} 
that for which parameters $(v,k,\lambda,\mu)$, does there exist a regular graph $G$ of  order $v$ and valency $2k$ such that its edge set can be decomposed into two sets of strongly regular graphs with the parameters $(v,k,\lambda,\mu)$ and there exists an automorphism of $G$ that swaps the two strongly regular graphs. 
Our answer is affirmative for $(v,k,\lambda,\mu)=(9^m,3^{m-1}(3^m-1),9^{m-1},3^{m-1}(3^{m-1}-1))$.   
\end{enumerate}
\end{remark}

\noindent {\bf Acknowledgments.}
The authors would like to thank the anonymous referees for their careful reading and suggestions to improve the paper.    
Hadi Kharaghani is supported by an NSERC Discovery Grant.  Sho Suda is supported by JSPS KAKENHI Grant Number 15K21075.


\end{document}